\documentclass[a4paper,12pt]{amsart}
\usepackage{graphicx}
\usepackage{amsmath,amssymb,amsthm,amsfonts}
\usepackage{amssymb}
\usepackage[active]{srcltx}
\usepackage[colorlinks=true]{hyperref}

\newtheorem{thm}{Theorem}[section]
\newtheorem{cor}[thm]{Corollary}
\newtheorem{lem}[thm]{Lemma}

\newtheorem{rem}[thm]{Remark}
\newtheorem{defn}[thm]{Definition}

\newcommand{\bremark}{\begin{rem} \textup}
\newcommand{\eremark}{\end{rem} }

\newcommand{\cuad}{{\sqcap\kern-.68em\sqcup}}

\newcommand{\R}{{\mathbb{R}}}
\newcommand{\N}{{\mathbb{N}}}

\renewcommand{\rho}{\varrho}
\renewcommand{\theta}{\vartheta}

\begin{document}


\parindent 0pc
\parskip 6pt
\overfullrule=0pt

\title[Qualitative properties of singular solutions ]{Qualitative properties of singular solutions to semilinear elliptic problems }

\author {F. Esposito*, A.Farina$^+$ and B. Sciunzi*}

\date{\today}

\date{\today}

\address{* Dipartimento di Matematica e Informatica, UNICAL,
Ponte Pietro  Bucci 31B, 87036 Arcavacata di Rende, Cosenza, Italy.}

\email{sciunzi@mat.unical.it}
\email{esposito@mat.unical.it}

\address{+ Universit\'e de Picardie Jules Verne, LAMFA, CNRS UMR
7352, 33, rue Saint-Leu 80039 Amines, France}

\email{alberto.farina@u-picardie.fr}




\maketitle

\date{\today}

\begin{abstract}
We consider positive singular solutions to semilinear elliptic
problems with possibly singular nonlinearity. We deduce symmetry and
monotonicity properties of the solutions via the moving plane
procedure.
\end{abstract}
\section{introduction}
The aim of the paper is to investigate symmetry and monotonicity properties of singular solutions to semilinear elliptic equations. We address the issue of problems involving singular nonlinearity. More precisely let us consider the
 problem
\begin{equation} \label{problem}
\begin{cases}
-\Delta u\,=f(x,u)& \text{in}\quad\Omega\setminus \Gamma  \\
u> 0 &  \text{in}\quad\Omega\setminus \Gamma  \\
u=0 &  \text{on}\quad\partial \Omega\,
\end{cases}
\end{equation}
where $\Omega$ is a bounded smooth domain of $\mathbb{R}^n$ with
$n\geq 2$. Our results will be obtained by means of the moving plane technique, see \cite{A,BN,GNN,serrin}. Such a technique can be performed in general domains providing partial monotonicity results near the boundary and symmetry when the domain is convex and symmetric.
For semplicity of exposition we assume directly in all the paper that $\Omega$
 is a convex domain which   is symmetric with
respect to the hyperplane $\{x_1=0\}$. The solution has a possible
singularity on the critical set   $\Gamma\subset \Omega$. Furthermore
in all the paper the nonlinearity $f$ will be assumed to be
uniformly locally Lipschitz continuous from above far from the
singular set. More precisely we state the following:

\begin{defn}[$h_f$]\label{djkfhsdjk}
\noindent  We  say that $f$ fulfills the condition $(h_f)$ if $f:
{\overline \Omega}\setminus \Gamma \times (0,+\infty) \rightarrow \R$ is a
continuous function such that for $0 \, < t \leq s\leq M$ and for any
compact set $K\subset {\overline \Omega} \setminus\Gamma$, it holds
\begin{equation}\nonumber
 f(x,s)-f(x,t)\leq C(K,M)(s-t)\qquad \text{for any}\quad x\in K\,,
\end{equation}
where $C(K,M)$ is a positive constant depending on $K$ and $M$.
Furthermore $f(\cdot,s)$ is non-decreasing in the $x_1$-direction in
$\Omega\cap\{x_1<0\}$ and symmetric with respect to the hyperplane
$\{x_1=0\}$.
\end{defn}
A typical example is provided by positive solutions to
\begin{equation}\label{singulareq}
-\Delta u\,=\frac{1}{u^\alpha} + g(u) \quad
\text{in}\quad\Omega\setminus \Gamma
\end{equation}
where $\alpha >0$ and $g$ is locally Lipschitz continuous. Such a
problem, in the case $\Gamma=\emptyset$, as been widely investigated
in the literature. We refer the readers to the pioneering work
\cite{crandall} and to
\cite{boccardo,CanDeg,gras1,ccmcan,nodt,lazer,stuart,oliva}. In
particular, by \cite{lazer}, it is known that
 solutions generally have no $H^1$-regularity up to the boundary. Therefore, having this example in mind,  the natural assumption in our paper is
$$u \in H_{loc}^1(\Omega \setminus \Gamma) \cap
C(\overline{\Omega} \setminus \Gamma)$$
and thus the equation is
understood in the following sense:
\begin{equation}\label{debil1}
\int_\Omega \nabla u\nabla \varphi\,dx\,=\,\int_\Omega f(x,u)\varphi\,dx\qquad\forall \varphi\in C^{1}_c(\Omega\setminus\Gamma)\,.
\end{equation}

\begin{rem}\label{standellest}
Note that, by the assumption $(h_f)$, the right hand
side in the equation of \eqref{problem} is locally bounded.
Therefore, by standard elliptic regularity theory, it follows that
$$u \in C_{loc}^{1,\alpha} (\Omega \setminus \Gamma),$$
\noindent where $0 < \alpha < 1$.
\end{rem}

Let us now state our main result

\begin{thm}\label{main}
Let $\Omega$ be a convex domain which is symmetric with respect to the hyperplane $\{x_1=0\}$ and let $u\in H^1_{loc}(\Omega\setminus\Gamma)\cap
C(\overline\Omega\setminus\Gamma)$ be a solution to \eqref{problem}.
Assume that $f$ fulfills $(h_f)$ (see Definition \ref{djkfhsdjk}). Assume also that $\Gamma$ is a point if $n=2$ while
 $\Gamma$ is closed and such that
$$\underset{\R^n}{\operatorname{Cap}_2}(\Gamma)=0,$$
if $n\geq 3$.
\noindent Then, if $\Gamma\subset\{x_1=0\}$, it follows that   $u$
is symmetric with respect to the hyperplane $\{x_1=0\}$ and
increasing in the $x_1$-direction in $\Omega\cap\{x_1<0\}$.
Furthermore
\[u_{x_1}>0\qquad \text{in}\quad \Omega\cap\{x_1<0\}\,.
\]
\end{thm}
\begin{rem}
Theorem \ref{main} is proved for convex domains. It will be clear from the proofs that this is only used to prove that $\partial\Omega\cap\{x_1=\lambda\}$ is discrete in dimension two while $\partial\Omega\cap\{x_1=\lambda\}$ has zero capacity for $n\geq3$. Therefore the result holds true more generally once that such an information is available. In all this cases we could  assume  that $\Omega$ is convex only in the $x_1$-direction.
\end{rem}
First results regarding the applicability of the moving plane procedure to the case of singular solutions go back to
 \cite{CLN2} (see also \cite{T}) where the case when the singular set is a single point is considered. We follow and improve here the technique in \cite{Dino}, where
 the case of a smooth $(n-2)$-dimensional singular set was considered in the case of locally Lipschitz continuous nonlinearity. Let us mention that
 the technique introduced in \cite{Dino} also works in the nonlocal context, see \cite{mps}.\\
 \noindent On the other hand, in the case $\Gamma=\emptyset$, symmetry and monotonicity properties of solutions to semilinear elliptic problems
 involving singular nonlinearities, have been studied in \cite{gras1,luigi}. Also in this direction our result is new and more general. In fact, while in \cite{gras1,luigi} it is necessary to restrict the attention to problems of the form \eqref{singulareq}, here we only need to consider nonlinearities that are locally Lipschitz continuous from above. Actually, all the nonlinearities of the form
 \[
 f(x,s)\,:= \, a_1(x)f_1(s)+a_2(x) f_2(s)\,,
 \]
where $f_1$ is an increasing continuous function in $[0,\infty)$, $f_2(\cdot),$ is locally Lipschitz continuous in $[0,\infty)$ and $ a_1, a_2 \in C^0(\overline{\Omega})$, $ \, a_1 \ge 0 $ on $ \overline{\Omega}$, satisfy our assumptions.

\noindent

The technique, as showed in \cite{Dino}, can be applied to study
singular solutions to the following Sobolev critical equation in
$\mathbb{R}^n, n\ge3,$
\begin{equation}
\label{problemubdd}
\begin{cases}
-\Delta u\,
= u^{2^*-1} & \text{in}\quad\mathbb{R}^n\setminus \Gamma  \\
u> 0 &  \text{in}\quad\mathbb{R}^n\setminus \Gamma.  \\
\end{cases}
\end{equation}
In \cite{Dino} it was considered the case of a closed critical set
$\Gamma$ contained in a compact smooth submanifold of dimension $d
\le n-2$ and a summability property of the solution at infinity was
imposed  (see also \cite{T} for the special case in which the
singular set $ \Gamma$ is reduced to a single point). Here we remove
both these restrictions and we prove the following:

\begin{thm}\label{main2}
Let $n\geq 3$ and let $u\in H^1_{loc}(\mathbb{R}^n\setminus\Gamma)$
be a solution to \eqref{problemubdd}. Assume that the solution $u$
has a non-removable\footnote{Here we mean that the solution $u$ does not admit a smooth extension all over the whole space. Namely it is not possible to find $\tilde u\in H^1_{loc}(\mathbb{R}^n)$ with $u\equiv\tilde u$ in $\mathbb{R}^n\setminus\Gamma$.} singularity in the singular set $\Gamma$, where
$\Gamma$ is a closed and proper subset of $\{x_1=0\}$ such that
$$\underset{\R^n}{\operatorname{Cap}_2}(\Gamma)=0. $$
\noindent Then, $u$ is symmetric with respect to the hyperplane $\{x_1=0\}$. \\
The same conclusion is true if the hyperplane $\{x_1=0\}$ is
replaced by any affine hyperplane.
\end{thm}

Some interesting consequences of the previous result are contained in the following

\begin{cor} \label{CorSim}
Let $n\geq 3$ and let $u\in H^1_{loc}(\mathbb{R}^n\setminus\Gamma)$
be a solution to \eqref{problemubdd} with a non-removable
singularity in the singular set $\Gamma$.
\item (i)  If $\Gamma =\{ x_0 \}$, then $u$ is radially symmetric with respect to $x_0$.
\item (ii)  If $\Gamma =\{ x_0, x_1 \}$, then $u$ has cylindrical symmetry with respect to the axis passing through $x_0$ and $x_1$. \\
More generally we have :
\item (iii) assume $ 1 \le k \le n-2 $ and suppose that $\Gamma$ is a closed subset of an affine $k-$dimensional subspace of $\R^n$. Then, up to isometry, the solution $u$ has the form $ u(x) = u(x_1,...,x_k, \vert x^{'} \vert)$,
where $x^{'} := (x_{k+1},...x_n)$ and $ \vert x^{'} \vert :=
\sqrt{x^2_{k+1} +...+ x^2_n}.$
\end{cor}

The following example shows that Theorem \ref{main2} and item (iii) of Corollary \ref{CorSim} are sharp for $n \ge 5$ and also that singular solutions exhibiting un unbounded critical set $\Gamma$ exist. \\\\
For $n \ge 5$ and $ 1 \le k < \frac{n-2}{2}$, $k$ integer, we set $
p=p(n) = \frac{n+2}{n-2} >1$ and $A = A(n,k) = [(\frac{n}{2}-k-1)
\frac{n}{2}]^{\frac{n-2}{4}}>0$. Then, the function $v(r) = A
r^{-\frac{2}{p(n)-1}} $ is a singular positive radial solution of $
- \Delta v = v^{p(n)}$ in $ \R^{n-k} \setminus \{0^{'}\}$, which is
smooth in  $ \R^{n-k} \setminus \{0^{'}\}$. Hence $u= u(x_1,...,x_n)
:= v(\vert x^{'} \vert)$ is a singular solution to
\eqref{problemubdd} in $\R^n \setminus \Gamma$, with $\Gamma$ given
by the $k-$dimensional subspace $\{ x_1=...=x_k=0 \} \subset \R^n$,
moreover $u \in C^{\infty}(\R^n \setminus \Gamma)$.

The remaining part of the paper is devoted to the proofs of our
results.

\section{Notations and preliminary results} \label{notations}

For a real number
$\lambda$ we set
\begin{equation}\label{eq:sn2}
\Omega_\lambda=\{x\in \Omega:x_1 <\lambda\}
\end{equation}
\begin{equation}\label{eq:sn3}
x_\lambda= R_\lambda(x)=(2\lambda-x_1,x_2,\ldots,x_n)
\end{equation}
which is the reflection through the hyperplane $T_\lambda :=\{ x_1=
\lambda\}$. Also let
\begin{equation}\label{eq:sn4}
a=\inf _{x\in\Omega}x_1.
\end{equation}

Since $\Gamma$ is compact and of zero capacity, $u$ is defined a.e. on $\Omega$ and Lebesgue measurable on $\Omega$.
Therefore the function \begin{equation}\label{eq:sn33}
u_\lambda := u \circ R_{\lambda}
\end{equation}
is Lebesgue measurable on $R_{\lambda}(\Omega)$. Similarly, $\nabla u$ and $\nabla u_{\lambda}$ are Lebesgue measurable on $\Omega$ and $R_{\lambda}(\Omega)$ respectively.


It is easy to see that, if
$\underset{\R^n}{\operatorname{Cap}_2}(\Gamma)=0$, then
$\underset{\R^n}{\operatorname{Cap}_2}(R_{\lambda}(\Gamma))=0$.
Another consequence of our assumptions
is that
$\underset{B^{\lambda}_{\epsilon}}{\operatorname{Cap}_2}(R_{\lambda}(\Gamma))=0$
for any open neighborhood $\mathcal{B}^{\lambda}_{\epsilon}$ of
$R_{\lambda}(\Gamma)$.
Indeed, recalling that $\Gamma$ is a point if $n=2$ while
 $\Gamma$ is closed with
$\underset{\R^n}{\operatorname{Cap}_2}(\Gamma)=0$
if $n\geq 3$ by assumption, it follows that
$$\underset{\mathcal{B}^{\lambda}_{\epsilon}}{\operatorname{Cap}_2}(R_{\lambda}(\Gamma)):=
\inf \left\{ \int_{\mathcal{B}^{\lambda}_{\epsilon}} |\nabla \varphi
|^2 dx < + \infty \; : \; \varphi \geq 1 \ \text{in} \
\mathcal{B}^{\lambda}_{\delta}, \; \varphi \in \
C^{\infty}_c(\mathcal{B}^{\lambda}_{\epsilon}) \right\}=0,$$
\noindent for some neighborhood $\mathcal{B}^{\lambda}_{\delta}
\subset \mathcal{B}^{\lambda}_{\varepsilon}$ of
$R_{\lambda}(\Gamma)$. From this, it follows that there exists
$\varphi_{\varepsilon} \in \
C^{\infty}_c(\mathcal{B}^{\lambda}_{\epsilon})$ such that
$\varphi_{\varepsilon} \geq 1$ in $\mathcal{B}^{\lambda}_{\delta}$
and $\displaystyle \int_{\mathcal{B}^{\lambda}_{\epsilon}} |\nabla
\varphi_{\varepsilon} |^2 dx < \varepsilon$.

Now we construct a function $\psi_{\varepsilon} \in C^{0,1}(\R^n, [0,1])$ such that
$\psi_{\varepsilon} = 1$ outside
$\mathcal{B}_{\varepsilon}^{\lambda}$, $\psi_{\varepsilon} = 0$ in
$\mathcal{B}_{\delta}^{\lambda}$ and
$$\int_{\R^n} |\nabla \psi_{\varepsilon} |^2 dx = \int_{\mathcal{B}^{\lambda}_{\epsilon}} |\nabla \psi_{\varepsilon}
|^2 dx < 4 \varepsilon.$$
\noindent To this end we consider the following Lipschitz continuous function

$$T_1(s)= \begin{cases}
1 &  \text{if}\quad s \le 0  \\
-2s + 1 & \text{if}\quad 0 \le s \le \frac{1}{2} \\
0 &  \text{if}\quad s \ge \frac{1}{2}  \\
\end{cases}$$

and we set
\begin{equation} \label{test1}
\psi_{\varepsilon} :=
T_1 \circ \varphi_{\varepsilon}
\end{equation}
where we have extended  $\varphi_{\varepsilon}$ by zero outside $\mathcal{B}_{\varepsilon}^{\lambda}$. Clearly $\psi_{\varepsilon} \in C^{0,1}(\R^n),  0 \le \psi_{\varepsilon} \le 1 $ and

$$\int_{\mathcal{B}^{\lambda}_{\epsilon}} |\nabla \psi_{\varepsilon}
|^2 dx \leq 4 \int_{\mathcal{B}^{\lambda}_{\epsilon}} |\nabla
\varphi_{\varepsilon} |^2 dx  < 4 \varepsilon.$$

Now we  set $\gamma_\lambda:= \partial \Omega \cap T_{\lambda}$. Recalling that $\Omega$ is convex, it is easy to  deduce that
$\gamma_\lambda$ is made of two points in dimension two. If else $n\geq3$ then it follows that
$\gamma_\lambda$ is a smooth manifold of dimension $n-2$.
 Note in fact that locally $\partial\Omega$ is the zero level set of a smooth function $g(\cdot)$ whose gradient is not parallel to the $x_1$-direction since $\Omega$ is convex. Then it is sufficient to observe that locally  $ \partial \Omega \cap T_{\lambda}\equiv\{g(\lambda,x')=0\}$ and use the \emph{implicit function theorem} exploiting the fact that $\nabla_{x'}g(\lambda,x')\neq0$.
 This implies that
$\underset{\R^n}{\operatorname{Cap}_2}(\gamma_\lambda)=0$, see  e.g.
\cite{evans}.
So, as before, $\underset{\mathcal{I}^\lambda_{\tau}}{\operatorname{Cap}_2}(\gamma_\lambda)=0$
for any open neighborhood of $\gamma_\lambda$ and then
there exists $\varphi_{\tau} \in C^{\infty}_c (\mathcal{I}^\lambda_{\tau})$ such that
$\varphi_{\tau} \geq 1$ in a
neighborhood $\mathcal{I}^\lambda_{\sigma}$ with
 $\gamma_\lambda\subset\mathcal{I}^\lambda_{\sigma} \subset
\mathcal{I}^\lambda_{\tau}$. As above, we set
\begin{equation}\label{test2}\phi_{\tau} :=T_1 \circ \varphi_{\tau}
\end{equation}
where we have extended  $\varphi_{\tau}$ by zero outside $\mathcal{I}^\lambda_{\tau}$. Then, $ \phi_{\tau} \in C^{0,1}(\R^n),  0 \le \phi_{\tau} \le 1, \phi_{\tau} = 1$ outside $\mathcal{I}^\lambda_{\tau}, \phi_{\tau} = 0$ in $\mathcal{I}^\lambda_{\sigma}$ and

$$\int_{\R^n} |\nabla \phi_\tau
|^2 dx = \int_{\mathcal{I}^\lambda_{\tau}} |\nabla \phi_\tau |^2 dx
\leq 4  \int_{\mathcal{I}^\lambda_{\tau}} |\nabla \varphi_\tau |^2
dx  < 4 \tau.$$

\section{Proof of Theorem \ref{main}}\label{mainproofsec}
In the following we will  exploit the fact that $u_\lambda$ is a solution to:

\begin{equation}\label{debil2}
\int_{R_\lambda(\Omega)} \nabla u_\lambda\nabla
\varphi\,dx\,=\,\int_{R_\lambda(\Omega)}
f(x_\lambda,u_\lambda)\varphi\,dx\qquad\forall \varphi\in
C^{1}_c(R_\lambda(\Omega)\setminus R_\lambda(\Gamma))\,
\end{equation}
and we also observe that, for any  $a<\lambda<0$, the function $w_\lambda\,:=\,u-u_\lambda$ satisfies  $0 \le w_\lambda^+ \le u $ a.e. on  ${\Omega_{\lambda}}$ and so $w_\lambda^+ \in L^2(\Omega_{\lambda})$, since  $u \in C^0(\overline{\Omega_{\lambda}})$. To proceed further, we need the following two results

\begin{lem}\label{leaiuto} Let $\lambda \in (a,0)$ be such that $R_\lambda (\Gamma) \cap {\overline {\Omega}} = \emptyset$ and consider the function
$$\varphi\,:= \begin{cases}
\, w_\lambda^+ \phi_{\tau}^2  & \text{in}\quad\Omega_\lambda, \\
 0 &  \text{in}\quad \R^n\setminus \Omega_\lambda,
\end{cases}$$
where $\phi_{\tau}$ is as in \eqref{test2}. Then, $ \varphi \in C^{0,1}_c (\Omega) \cap C^{0,1}_c (R_{\lambda}(\Omega)), \, \varphi$ has compact support contained in $(\Omega \setminus \Gamma) \cap (R_{\lambda}(\Omega) \setminus R_{\lambda}(\Gamma)) \cap \{ x_N \le \lambda \}$ and
\begin{equation}\label{gradvarphi}
\nabla \varphi = \phi_\tau^2 (\nabla w_\lambda \chi_{supp(w_\lambda^+) \cap supp(\varphi)}) + 2 \phi_\tau
(w_\lambda^+ \chi_{supp(\varphi)}) \nabla \phi_\tau \quad {\text {a.e. on }} \, \, \Omega \cup R_{\lambda}(\Omega).
\end{equation}
If $\lambda \in (a,0)$ is such that $R_\lambda (\Gamma) \cap {\overline {\Omega}} \neq \emptyset$, the same conclusions hold true for the function
$$\varphi\,:= \begin{cases}
\, w_\lambda^+ \psi_\varepsilon^2 \phi_{\tau}^2  & \text{in}\quad\Omega_\lambda, \\
 0 &  \text{in}\quad \R^n \setminus \Omega_\lambda,
\end{cases}$$
where $\psi_\varepsilon$ is defined as in \eqref{test1} and $\phi_{\tau}$
as in \eqref{test2}. Furthermore, a.e. on $ \Omega \cup R_{\lambda}(\Omega)$,
\begin{equation}\label{gradvarphi2}
\nabla \varphi = \psi_\varepsilon^2 \phi_\tau^2 (\nabla w_\lambda \chi_{supp(w_\lambda^+) \cap supp(\varphi)}) +
2 (w_\lambda^+ \chi_{supp(\varphi)}) (\psi_\varepsilon^2 \phi_\tau \nabla \phi_\tau + \psi_\varepsilon \phi_\tau^2  \nabla  \psi_\varepsilon).
\end{equation}
{In particular, $\varphi \in C^{0,1}(\overline{\Omega_\lambda})$, $\varphi_{\vert _{\partial \Omega_\lambda}} = 0$   and so $ \varphi \in H^1_0 (\Omega_\lambda)$. }

\end{lem}

\begin{proof}
Let us consider the case when $\lambda \in (a,0)$ is such that $R_\lambda (\Gamma) \cap {\overline {\Omega}} \neq \emptyset$ (the other case being similar and easier). We first prove that for every $x \in \Omega$  there is an open ball $B_x$ centered at $x$, such that $ \overline{B_x} \subset \Omega$ and $\varphi \in C^{0,1}(\overline{B_x})$,
and then that there exists $\eta >0$ such that $supp(\varphi) $ is contained in the compact set $\{ x \in \Omega : dist(x, \partial \Omega) \ge \eta \} \cap \{ x_N \le \lambda\} \cap (\R^n \setminus V) \subset (\Omega \setminus \Gamma) \cap (R_{\lambda}(\Omega) \setminus R_{\lambda}(\Gamma))$, where $V$ is any open set contained in the neighborhood $\mathcal{B}_{\delta}^{\lambda}$ appearing in the construction of $ \psi_\varepsilon $.

If $x \in \Omega \cap \{ x_N > \lambda\} $ then $ \varphi \equiv 0$ in an open neighbourhood of $x$ and so  $\varphi \in C^{0,1}(\overline{B_x})$ for a suitable ball $B_x$.
If $x \in \Omega \cap T_{\lambda}$ then we can find a small open ball $B_x \subset \Omega $ such that $B_x \cap (\partial \Omega \cup R_\lambda (\Gamma) ) =\emptyset$. Therefore, both $u$ and $u_{\lambda}$ belong to $ C^1(\overline {B_x} \cap \{ x_N \le  \lambda\})$ and so, $\varphi \in C^{0,1}(\overline {B_x}\cap \{ x_N \le \lambda\})$, thanks to the lipschitz character of $ \phi_{\tau}$ and $\psi_\varepsilon$.
On the other hand we also have that $\varphi \equiv 0 $ on $\overline {B_x}\cap T_{\lambda}$, by definition of $w_{\lambda}$. Thus  $\varphi \in C^{0,1}(\overline{B_x})$ and we are done also in this case.
If $ x \in R_\lambda (\Gamma) \cap \Omega$ then $ \varphi \equiv 0$ in an open neighbourhood of $x$ by definition of
$\psi_\varepsilon$ and so  $\varphi \in C^{0,1}(\overline{B_x})$ for a suitable ball $B_x$.
Finally, if $x \in \Omega_{\lambda} \setminus R_\lambda (\Gamma)$ then, as before, we can find a small open ball $B_x$ such that $\overline {B_x} \subset  \Omega_{\lambda} \setminus R_\lambda (\Gamma)$.  In this case, both $u$ and $u_{\lambda}$ belong to $ C^1(\overline {B_x})$. This yields $w_{\lambda} \in C^{0,1}(\overline{B})$ and so is $\varphi$, again thanks to the lipschitz character of $ \phi_{\tau}$ and $\psi_\varepsilon$.

To prove the second part of the claim we observe that $ \varphi \equiv 0$ on $\Omega \setminus \Omega_{\lambda}$ and that, for any point $x$ of the compact set $ (\partial \Omega) \cap \{ x_N \le \lambda\}$ there is a small open ball $B_x$, centered at $x$, such that $ \varphi = 0 $ on $B_x \cap \Omega$. The latter clearly holds for any point of $\gamma_{\lambda}$, by definition of  $\phi_{\tau}$, and for any point of $\partial \Omega \cap R_\lambda (\Gamma) $, by definition of $\psi_\varepsilon$. It is also true for any $x \in (\partial \Omega) \cap \{ x_N  <\lambda\}$, since $ u - u_{\lambda} $ is well-defined, continuous and negative on the set  $[(\partial \Omega) \cap \{ x_N  <\lambda\}] \setminus R_\lambda (\Gamma)$. The arguments above immediately yield that $\varphi \in  C^{0,1}_{c}(\Omega) $ and the formula \eqref{gradvarphi2}. A similar argument also shows that $\varphi \in  C^{0,1}_{c}(R_{\lambda}(\Omega))$.


To compute $\nabla \varphi$ we also took into consideration the Remark \ref{standellest}.

\end{proof}

\begin{lem}\label{leaiuto2}
Under the assumptions of Theorem \ref{main}, let $a<\lambda<0$. Then $ w_\lambda^+\in H^1_0(\Omega_\lambda)$ and
\begin{equation}\nonumber
\int_{\Omega_\lambda}|\nabla w_\lambda^+|^2\,dx\leq c(f,{\vert \Omega \vert}, \|u\|_{L^\infty(\Omega_\lambda)}),
\end{equation}
{where $\vert \Omega \vert$ denotes the $n-$dimensional Lebesgue measure of $ \Omega$. }
\end{lem}

\begin{proof}

We first prove that $\nabla w_\lambda \chi_{supp(w_\lambda^+)} \in L^2(\Omega_{\lambda})$ and then that the distributional gradient of $  w_\lambda^+$ is given by $\nabla w_\lambda  \chi_{supp(w_\lambda^+)} $. We do this only for the case in which $ \lambda $ is such that $R_\lambda (\Gamma) \cap {\overline {\Omega}} \neq \emptyset$, the other case being similar and easier. For $\psi_\varepsilon$ as in \eqref{test1} and $\phi_{\tau}$
as in \eqref{test2}, we consider the function $\varphi$ defined in Lemma \ref{leaiuto}. In view of the properties of $\varphi$, stated in Lemma \ref{leaiuto}, and a standard density argument, we can use $\varphi$ as test function in
\eqref{debil1} and \eqref{debil2} so that, subtracting, we get

\begin{equation}\nonumber
\begin{split}
\int_{\Omega_\lambda}|\nabla w_\lambda \chi_{supp(w_\lambda^+)} |^2\psi_\varepsilon^2
\phi_{\tau}^2 \,dx&=
 -2\int_{\Omega_\lambda}\nabla w_\lambda \nabla \psi_\varepsilon w_\lambda^+ \psi_\varepsilon \phi_{\tau}^2
 \,dx - 2\int_{\Omega_\lambda}\nabla w_\lambda \nabla \phi_\tau w_\lambda^+ \psi_\varepsilon^2 \phi_{\tau}
 \,dx\\
 &+\int_{\Omega_\lambda} \left(f(x,u)-f(x_\lambda,u_\lambda)\right) w_\lambda^+ \psi_\varepsilon^2
\phi_{\tau}^2 \,dx\,\\
  &\leq -2\int_{\Omega_\lambda}\nabla w_\lambda \nabla \psi_\varepsilon w_\lambda^+ \psi_\varepsilon \phi_{\tau}^2
 \,dx - 2\int_{\Omega_\lambda}\nabla w_\lambda \nabla \phi_\tau w_\lambda^+ \psi_\varepsilon^2 \phi_{\tau}
\,dx\\
 &+\int_{\Omega_\lambda} (f(x,u)-f(x,u_\lambda)) w_\lambda^+ \psi_\varepsilon^2
\phi_{\tau}^2 \,dx.\\
\end{split}
\end{equation}
Here we also used the monotonicity properties of $f(\cdot,s)$, see $(h_f)$.
Exploiting Young's inequality we get that
\begin{equation}\label{stima1}
\begin{split}
\int_{\Omega_\lambda}|\nabla w_\lambda \chi_{supp(w_\lambda^+)} |^2\psi_\varepsilon^2
\phi_{\tau}^2 \,dx &\leq \frac{1}{4}\int_{\Omega_\lambda}|\nabla
w_\lambda \chi_{supp(w_\lambda^+)}|^2\psi_\varepsilon^2 \phi_{\tau}^2 \,dx
+4\int_{\Omega_\lambda}|\nabla \psi_\varepsilon|^2 (w_\lambda^+)^2
\phi_{\tau}^2
 \,dx \\
&+\frac{1}{4}\int_{\Omega_\lambda}|\nabla
w_\lambda \chi_{supp(w_\lambda^+)}|^2\psi_\varepsilon^2 \phi_{\tau}^2 \,dx +
4\int_{\Omega_\lambda}|\nabla \phi_\tau|^2 (w_\lambda^+)^2
\psi_\varepsilon^2\, dx\\
 &+\int_{\Omega_\lambda} (f(x,u)-f(x,u_\lambda)) w_\lambda^+ \psi_\varepsilon^2
\phi_{\tau}^2 \,dx.\\
\end{split}
\end{equation}
Now we observe that the last integral is actually computed on the set
$\{ x \in \Omega_{\lambda} \setminus R_\lambda (\Gamma) : u(x) > u_\lambda(x) >0 \} \subset {\overline {\Omega_\lambda}} \subset \overline{\Omega} \setminus \Gamma$ and so, we can apply condition $(h_f)$ with the compact set $ K= {\overline {\Omega_\lambda}} $ and $ M = \|u\|_{L^\infty(\Omega_\lambda)}$.  We get therefore that
\begin{equation}\label{stima2}
\begin{split}
\int_{\Omega_\lambda} (f(x,u)-f(x,u_\lambda)) w_\lambda^+ \psi_\varepsilon^2
\phi_{\tau}^2 \,dx
& \le c(f, \|u\|_{L^\infty(\Omega_\lambda)})\int_{\Omega_\lambda}(w_\lambda^+)^2 \psi_\varepsilon^2
\phi_{\tau}^2\,dx
\end{split}
\end{equation}
and so, from \eqref{stima1}, we infer that
\begin{equation}\label{stima3}
\begin{split}
\int_{\Omega_\lambda}|\nabla w_\lambda \chi_{supp(w_\lambda^+)}|^2\psi_\varepsilon^2
\phi_{\tau}^2 \,dx &\leq 8 \int_{\Omega_\lambda}|\nabla \psi_\varepsilon|^2 (w_\lambda^+)^2
\phi_{\tau}^2
 \,dx + 8 \int_{\Omega_\lambda}|\nabla \phi_\tau|^2 (w_\lambda^+)^2
\psi_\varepsilon^2\, dx\\
 &+2c(f, \|u\|_{L^\infty(\Omega_\lambda)})\int_{\Omega_\lambda}(w_\lambda^+)^2 \psi_\varepsilon^2
\phi_{\tau}^2\,dx.\\
\end{split}
\end{equation}
Taking into account the properties of
$\psi_\varepsilon$ and $\phi_\tau$, we see that
\begin{equation}\label{bbeg}
\int_{\Omega_\lambda}|\nabla \psi_\varepsilon|^2\,dx=
\int_{\Omega_\lambda\cap(\mathcal B_\varepsilon^\lambda\setminus
\mathcal B_{\delta}^\lambda)}|\nabla \psi_\varepsilon|^2\,dx < 4
\varepsilon,
\end{equation}
\begin{equation}\label{bbeg}
\int_{\Omega_\lambda}|\nabla \phi_\tau|^2\,dx=
\int_{\Omega_\lambda\cap(\mathcal I^\lambda_\tau \setminus \mathcal
I^\lambda_{\sigma})}|\nabla \phi_\tau|^2\,dx < 4 \tau,
\end{equation}
which combined with $0\leq w_\lambda^+\leq u$,  immediately lead to
\begin{equation}\nonumber
\begin{split}
\int_{\Omega_\lambda}|\nabla w_\lambda \chi_{supp(w_\lambda^+)}|^2\psi_\varepsilon^2
\phi_{\tau}^2 \,dx \leq 32(\varepsilon +\tau) ||u||^2_{L^\infty
(\Omega_\lambda)} +2 c(f, \|u\|_{L^\infty(\Omega_\lambda)}) ||u||^2_{L^\infty
(\Omega_\lambda)} \vert \Omega \vert \,.
\end{split}
\end{equation}
By Fatou Lemma, as $\varepsilon$ and $\tau$ tend to zero, we deduce that $\nabla w_\lambda \chi_{supp(w_\lambda^+)}
\in L^2(\Omega_{\lambda})$.  To conclude we note that $ \varphi \to  w_\lambda^+$ in $\L^2(\Omega)$,  as $\varepsilon$ and $\tau$ tend to zero, by definition of $ \varphi $. Also, $\nabla \varphi \to \nabla w_\lambda \chi_{supp(w_\lambda^+)}$ in $L^2(\Omega_{\lambda})$, by \eqref{gradvarphi2}. Therefore, $\nabla w_\lambda \chi_{supp(w_\lambda^+)}$ is the distributional gradient of $\nabla w_\lambda^+ $ and $ w_\lambda^+$ in $H^1_0(\Omega_\lambda)$, since {$\varphi \in H^1_0(\Omega_\lambda)$} again by Lemma \ref{leaiuto}.  Which concludes the proof.


\end{proof}

\begin{proof}[Proof of Theorem \ref{main}]

We define
\begin{equation}\nonumber
\Lambda_0=\{a<\lambda<0 : u\leq
u_{t}\,\,\,\text{in}\,\,\,\Omega_t\setminus
R_t(\Gamma)\,\,\,\text{for all $t\in(a,\lambda]$}\}
\end{equation}
and to start with the moving plane procedure, we have to prove
that

\emph{Step 1 :  $\Lambda_0 \neq \emptyset$}. Fix a $ \lambda_0 \in (a,0)$ such that
$R_{\lambda_0}(\Gamma) \subset \Omega^c$, then for every $a< \lambda < \lambda_0$, we also have that
$R_\lambda(\Gamma)\subset \Omega^c$. 
For any $ \lambda$ in this set we consider, on the domain $\Omega$, the function $\varphi\,:=\, w_\lambda^+ \phi_{\tau}^2 \chi_{\Omega_\lambda},$ where $\phi_{\tau}$ is as in \eqref{test2} and we proceed as in the proof of Lemma \ref{leaiuto2}. That is, by Lemma \ref{leaiuto} and a density argument, we can use $\varphi$ as test function in
\eqref{debil1} and \eqref{debil2} so that, subtracting, we get

\begin{equation}\nonumber
\begin{split}
\int_{\Omega_\lambda}|\nabla w_\lambda^+|^2 \phi_{\tau}^2 \,dx&=
 - 2\int_{\Omega_\lambda} \nabla w_\lambda^+\nabla \phi_\tau w_\lambda^+ \phi_{\tau}
 \,dx +\int_{\Omega_\lambda} \left(f(x,u)-f(x_\lambda,u_\lambda)\right)
w_\lambda^+ \phi_{\tau}^2 \,dx\,\\
 &\leq- 2\int_{\Omega_\lambda} \nabla w_\lambda^+\nabla \phi_\tau w_\lambda^+  \phi_{\tau}
 \,dx +\int_{\Omega_\lambda} (f(x,u)-f(x,u_\lambda)) w_\lambda^+ \phi_{\tau}^2 \,dx.
 \end{split}
\end{equation}

Exploiting Young's inequality and the assumption $(h_f)$, with $K ={\overline {\Omega_{\lambda_0}}}$ and $ M = ||u||^2_{L^\infty
(\Omega_{\lambda_0})}$,  we then get that
\begin{equation}\nonumber
\begin{split}
\int_{\Omega_\lambda}|\nabla w_\lambda^+|^2 \phi_{\tau}^2 \,dx &\leq
\frac{1}{2}\int_{\Omega_\lambda}|\nabla w_\lambda^+|^2 \phi_{\tau}^2
\,dx + 2\int_{\Omega_\lambda}|\nabla \phi_\tau|^2 (w_\lambda^+)^2 dx\\
&+ c(f,\|u\|_{L^\infty(\Omega_{\lambda_0})})
\int_{\Omega_\lambda}(w_\lambda^+)^2 \phi_{\tau}^2\,dx.
\end{split}
\end{equation}
Taking into account the properties of $\phi_\tau$, we see that
\begin{equation}\label{bbeg}
\int_{\Omega_\lambda}|\nabla \phi_\tau|^2 (w_\lambda^+)^2 dx \leq \|
u \|_{L^\infty (\Omega_\lambda)}^2 \int_{\Omega_\lambda \cap (\mathcal
I^\lambda_\tau \setminus \mathcal I^\lambda_{\sigma})}|\nabla
\phi_\tau|^2\,dx\leq 4 \| u \|_{L^\infty(\Omega_\lambda)}^2 \cdot
\tau.
\end{equation}
We therefore deduce that
\begin{equation}\nonumber
\begin{split}
\int_{\Omega_\lambda}|\nabla w_\lambda^+|^2 \phi_{\tau}^2 \,dx \leq
16 ||u||_{L^\infty (\Omega_\lambda)} \cdot \tau +
2 c(f,\|u\|_{L^\infty(\Omega_{\lambda_0})})
\int_{\Omega_\lambda}(w_\lambda^+)^2 \phi_{\tau}^2\,dx.
\end{split}
\end{equation}
By Fatou Lemma, as $\tau$ tend to, zero we have
\begin{equation}\label{ff}
\begin{split}
\int_{\Omega_\lambda}|\nabla w_\lambda^+|^2  \,dx &\leq
2 c(f,\|u\|_{L^\infty(\Omega_{\lambda_0})}) \int_{\Omega_\lambda}(w_\lambda^+)^2 \, dx \\
& \leq 2 c(f,\|u\|_{L^\infty(\Omega_{\lambda_0})}) c_p^2(\Omega_\lambda)
\int_{\Omega_\lambda}|\nabla w_\lambda^+|^2  \,dx,
  \end{split}
\end{equation}
where $c_p(\cdot)$ is the Poincar\'e constant (in the Poincar\'e inequality in $H^1_0(\Omega_\lambda)$).
Since $c_p^2(\Omega_\lambda) \to 0$ as $ \lambda \to a$, we can find $ \lambda_1 \in (a, \lambda_0)$, such that
$$\forall \lambda \in (a, \lambda_1) \qquad  2 c(f,\|u\|_{L^\infty(\Omega_{\lambda_0})}) c_p^2(\Omega_\lambda)< \frac{1}{2}\,, $$ so that by \eqref{ff}, we
deduce that
$$ \forall \lambda \in (a, \lambda_1) \qquad \int_{\Omega_\lambda}|\nabla w_\lambda^+|^2  \,dx \leq 0,$$ proving
that $u \leq u_\lambda$ in $\Omega_\lambda \setminus R_\lambda(\Gamma)$ for $\lambda$ close to $a$, which implies the desired conclusion  $\Lambda_0 \neq \emptyset$.

\noindent Now we can set
 \begin{equation}\nonumber
\lambda_0=\sup\,\Lambda_0.
\end{equation}

\emph{Step 2: here we show that $\lambda_0=0$.} To this end we assume that
$\lambda_0<0$ and we reach a contradiction by proving that $u\leq
u_{\lambda_0+\nu}$ in $\Omega_{\lambda_0+\nu}\setminus
R_{\lambda_0+\nu}(\Gamma)$ for any $0<\nu<\bar\nu$ for some small
$\bar\nu>0$. By continuity we know that $u\leq u_{\lambda_0}$ in
$\Omega_{\lambda_0}\setminus R_{\lambda_0}(\Gamma)$.
{Since $\Omega$ is convex in the $x_1-$direction and the set $ R_{\lambda_0}(\Gamma)$ lies in the hyperplane of equation $\{\, x_1 = - 2 \lambda_0 \, \}$, we see that $\Omega_{\lambda_0}\setminus R_{\lambda_0}(\Gamma)$ is open and connected. Therefore, by the strong maximum principle we deduce that
$u< u_{\lambda_0}$ in $\Omega_{\lambda_0}\setminus R_{\lambda_0}(\Gamma)$ (here we have also used that
$ u, u_{\lambda_0}\in C^1(\Omega_{\lambda_0}\setminus R_{\lambda_0}(\Gamma)) $ by  Remark \ref{standellest}, as well as the assumption $(h_f)$.) }

{Now, note that for $K\subset \Omega_{\lambda_0}\setminus R_{\lambda_0}(\Gamma)$, there is $\nu=\nu(K,\lambda_0)>0$, sufficiently small, such that $K \subset \Omega_{\lambda} \setminus R_{\lambda}(\Gamma)$ for every $ \lambda \in [\lambda_0, \lambda_0 + \nu].$ Consequently $u$ and $u_{\lambda}$ are well defined on $K$ for every $ \lambda \in [\lambda_0, \lambda_0 + \nu].$ Hence, by the uniform continuity of the function $ g(x,\lambda) := u(x) - u(2\lambda-x_1,x^{'}) $ on the compact set $K \times [\lambda_0, \lambda_0 + \nu]$ we can ensure that  $K \subset \Omega_{\lambda_0 + \nu} \setminus R_{\lambda_0 + \nu}(\Gamma)$ and
$u< u_{\lambda_0+\nu}$ in $K$ for any $0 \le \nu<\bar\nu$, for some $\bar\nu = \bar\nu(K,\lambda_0)>0$ small. Clearly we can also assume that $\bar\nu < \frac{\vert \lambda_0 \vert}{4}. $}


{Let us} consider $\psi_\varepsilon$ constructed
in such a way that it vanishes in a neighborhood  of $R_{\lambda_0 +
\nu}(\Gamma)$ and $\phi_{\tau}$ constructed in such a way it
vanishes in a neighborhood  of $\gamma_{\lambda_0+\nu}=\partial
\Omega \cap T_{\lambda_0+ \nu}$. {As swown in the proof of lemma \ref{leaiuto2}, the functions}
$$\varphi\,:= \begin{cases}
\, \, w_{\lambda_0+\nu}^+\psi_\varepsilon^2 \phi_{\tau}^2 \, & \text{in}\quad {\Omega_{\lambda_0+\nu}} \\
 0 &  \text{in}\quad {\R^n \setminus \Omega_{\lambda_0+\nu}}
\end{cases}$$
{are such that $ \varphi \to  w_{\lambda_0+\nu}^+$ in $H^1_0(\Omega_{\lambda_0+\nu})$,  as $\varepsilon$ and $\tau$ tend to zero. Moreover, $\varphi \in C^{0,1}(\overline{\Omega_{\lambda_0+\nu}})$ and $\varphi_{\vert _{\partial \Omega_{\lambda_0+\nu}}} = 0$, by Lemma \ref{leaiuto}, and $ \varphi = 0 $ on an open neighborhood of $K$, by the above argument.
Therefore, $ \varphi \in H^1_0 (\Omega_{\lambda_0+\nu} \setminus K)$ and thus, also $w_{\lambda_0+\nu}^+$ belongs to  $H^1_0(\Omega_{\lambda_0+\nu} \setminus K)$. We also note that $ \nabla w_{\lambda_0+\nu}^+ = 0$ on an open neighborhood of $K$.}

\noindent {Now we argue as in Lemma \ref{leaiuto2} and we plug $\varphi$ as test function in
\eqref{debil1} and \eqref{debil2} so that, by subtracting, we get  }

{\begin{equation}\nonumber
\begin{split}
\int_{\Omega_{\lambda_0+\nu}} |\nabla
w_{\lambda_0+\nu}^+|^2\psi_\varepsilon^2 \phi_{\tau}^2\,dx&\leq
-2\int_{\Omega_{\lambda_0+\nu}}\nabla w_{\lambda_0+\nu} \nabla \psi_\varepsilon w_{\lambda_0+\nu}^+ \psi_\varepsilon \phi_{\tau}^2 \,dx \\
&- 2\int_{\Omega_{\lambda_0+\nu}}\nabla
w_{\lambda_0+\nu} \phi_\tau w_{\lambda_0+\nu}^+
\psi_\varepsilon^2 \phi_{\tau}
 \,dx\\
 &+ \int_{\Omega_{\lambda_0+\nu}} (f(x,u)-f(x,u_\lambda)) w_{\lambda_0+\nu}^+ \psi_\varepsilon^2
\phi_{\tau}^2 \,dx\\
 \end{split}
\end{equation}
where we also use the monotonicity of $f(\cdot,s)$ in the
$x_1$-direction.
Therefore, taking into account the properties of $ w_{\lambda_0+\nu}^+$ and $ \nabla w_{\lambda_0+\nu}^+$ we also have }

\begin{equation}\nonumber
\begin{split}
\int_{\Omega_{\lambda_0+\nu}\setminus K}|\nabla
w_{\lambda_0+\nu}^+|^2\psi_\varepsilon^2 \phi_{\tau}^2\,dx&\leq
-2\int_{\Omega_{\lambda_0+\nu}\setminus K}\nabla w_{\lambda_0+\nu}^+\nabla \psi_\varepsilon w_{\lambda_0+\nu}^+ \psi_\varepsilon \phi_{\tau}^2 \,dx \\
&- 2\int_{\Omega_{\lambda_0+\nu}\setminus K}\nabla
w_{\lambda_0+\nu}^+\nabla \phi_\tau w_{\lambda_0+\nu}^+
\psi_\varepsilon^2 \phi_{\tau}
 \,dx\\
 &+ \int_{\Omega_{\lambda_0+\nu}\setminus K} (f(x,u)-f(x,u_\lambda)) w_{\lambda_0+\nu}^+ \psi_\varepsilon^2
\phi_{\tau}^2 \,dx.\\
  \end{split}
\end{equation}
Furthermore, since $f$ is locally uniformly
Lipschitz continuous from above, we deduce that

\begin{equation}\label{jkdfh}
\begin{split}
\int_{\Omega_{\lambda_0+\nu}\setminus K}|\nabla
w_{\lambda_0+\nu}^+|^2\psi_\varepsilon^2 \phi_{\tau}^2\,dx&\leq
2\int_{\Omega_{\lambda_0+\nu}\setminus K}|\nabla w_{\lambda_0+\nu}^+| |\nabla \psi_\varepsilon| w_{\lambda_0+\nu}^+ \psi_\varepsilon \phi_{\tau}^2 \,dx \\
&+ 2\int_{\Omega_{\lambda_0+\nu}\setminus K} |\nabla
w_{\lambda_0+\nu}^+| |\nabla \phi_\tau| w_{\lambda_0+\nu}^+
\psi_\varepsilon^2 \phi_{\tau}
 \,dx\\
 &+ c(f,||u||_{L^\infty(\Omega_{\lambda_0 + {\frac{\vert \lambda_0 \vert}{4}} })}) \int_{\Omega_{\lambda_0+\nu}\setminus K}  (w_{\lambda_0+\nu}^+)^2 \psi_\varepsilon^2
\phi_{\tau}^2 \,dx.\\
  \end{split}
\end{equation}

Now, as in the proof of Lemma \ref{leaiuto2}, we use Young's
inequality to deduce that
\begin{equation}
\begin{split}
\int_{\Omega_{\lambda_0+\nu}\setminus K}
\vert \nabla w_{\lambda_0+\nu}^+|^2\psi_\varepsilon^2
\phi_{\tau}^2 \,dx &\leq 8 \int_{\Omega_{\lambda_0+\nu}\setminus K}|\nabla \psi_\varepsilon|^2 (w_{\lambda_0+\nu}^+)^2 \phi_{\tau}^2
\,dx \\
&+ 8 \int_{\Omega_{\lambda_0+\nu}\setminus K}|\nabla \phi_\tau|^2
(w_{\lambda_0+\nu}^+)^2
\psi_\varepsilon^2\, dx\\
 &+2c(f, ||u||_{L^\infty(\Omega_{\lambda_0+ {\frac{\vert \lambda_0 \vert}{4}} })}) \int_{\Omega_{\lambda_0+\nu} \setminus K}(w_\lambda^+)^2 \psi_\varepsilon^2
\phi_{\tau}^2\,dx,\\
\end{split}
\end{equation}
 which in turns yields
 \begin{equation}
\begin{split}
\int_{\Omega_{\lambda_0+\nu}\setminus K}
\vert \nabla w_{\lambda_0+\nu}^+|^2\psi_\varepsilon^2
\phi_{\tau}^2 \,dx &\leq 32 ||u||^2_{L^\infty(\Omega_{\lambda_0+ \bar\nu})} (\epsilon +\tau)\\
&+2c(f, ||u||_{L^\infty(\Omega_{\lambda_0+ {\frac{\vert \lambda_0
\vert}{4}} })}) \int_{\Omega_{\lambda_0+\nu} \setminus
K}(w_\lambda^+)^2 \psi_\varepsilon^2
\phi_{\tau}^2\,dx.\\
\end{split}
\end{equation}
Passing to the limit, as $(\epsilon, \tau) \to (0,0),$ in the latter
we get

\begin{equation}\label{sdjfhsfskl}
\begin{split}
&\int_{\Omega_{\lambda_0+\nu}\setminus K}|\nabla
w_{\lambda_0+\nu}^+|^2\,dx \leq {2 c(f, \|u\|_{L^\infty(\Omega_{\lambda_0 +{\frac{\vert \lambda_0 \vert}{4}} })}) }
\int_{\Omega_{\lambda_0+\nu}\setminus K}  (w_{\lambda_0+\nu}^+)^2\,dx\\
&\leq {2 c(f, \|u\|_{L^\infty(\Omega_{\lambda_0 +{\frac{\vert \lambda_0 \vert}{4}} })}) }
c_p^2(\Omega_{\lambda_0+\nu}\setminus K)
\int_{\Omega_{\lambda_0+\nu}\setminus K}   |\nabla
w_{\lambda_0+\nu}^+|^2\,dx\, ,
  \end{split}
\end{equation}
where $c_p(\cdot)$ is the Poincar\'e constant (in the Poincar\'e inequality in $H^1_0(\Omega_{\lambda_0+\nu}\setminus K)$).
{Now we recall that $ c_p^2(\Omega_{\lambda_0+\nu}\setminus K) \le Q(n) \vert \Omega_{\lambda_0+\nu}\setminus K \vert ^{\frac{2}{N}} $, where $Q=Q(n)$ is a positive constant depending only on the dimension $n$, and therefore, by summarizing, we have proved that for every compact set $K\subset \Omega_{\lambda_0}\setminus R_{\lambda_0}(\Gamma)$ there is a small $\bar\nu = \bar\nu(K, \lambda_0) \in (0, \frac{\vert \lambda_0 \vert}{4})$ such that for every $ 0 \le \nu < \bar\nu$ we have
\begin{equation}\label{sdjfhsfskl2}
\begin{split}
&\int_{\Omega_{\lambda_0+\nu}\setminus K}|\nabla
w_{\lambda_0+\nu}^+|^2\,dx \leq {2 c(f, \|u\|_{L^\infty(\Omega_{\lambda_0 +{\frac{\vert \lambda_0 \vert}{4}} })}) }Q(n) \vert \Omega_{\lambda_0+\nu}\setminus K \vert ^{\frac{2}{N}}
\int_{\Omega_{\lambda_0+\nu}\setminus K}   |\nabla
w_{\lambda_0+\nu}^+|^2\,dx.
\end{split}
\end{equation}
Now we first fix a compact $K\subset \Omega_{\lambda_0}\setminus R_{\lambda_0}(\Gamma)$ such that
$$\vert \Omega_{\lambda_0}\setminus K \vert^{\frac{2}{N}} < [ 20 c(f, \|u\|_{L^\infty(\Omega_{\lambda_0 +{\frac{\vert \lambda_0 \vert}{4}} })}) Q(n)]^{-1},$$ this is possible since $ \vert  R_{\lambda_0}(\Gamma)\vert  =0$ by the assumption on $ \Gamma$, and then we take $ \bar\nu_0 < \bar\nu $ such that for every $ 0 \le \nu < \bar\nu_0$ we have
$\vert \Omega_{\lambda_0 + \nu} \setminus \Omega_{\lambda_0}
\vert^{\frac{2}{N}}< [ 20 c(f, \|u\|_{L^\infty(\Omega_{\lambda_0
+{\frac{\vert \lambda_0 \vert}{4}} })}) Q(n)]^{-1}$. Inserting those
informations into \eqref{sdjfhsfskl2} we immediately get that
\begin{equation}\label{sdjfhsfskl3}
\begin{split}
&\int_{\Omega_{\lambda_0+\nu}\setminus K}|\nabla
w_{\lambda_0+\nu}^+|^2\,dx  < \frac{1}{2}
\int_{\Omega_{\lambda_0+\nu}\setminus K}   |\nabla
w_{\lambda_0+\nu}^+|^2\,dx
\end{split}
\end{equation}
and so $\nabla w_{\lambda_0+\nu}^+$ on $ \Omega_{\lambda_0+\nu} \setminus K$ for every $ 0 \le \nu < \bar\nu_0$. On the other hand, we recall that $\nabla w_{\lambda_0+\nu}^+$ on an open neighbourhood of $K$ for every $ 0 \le \nu < \bar\nu$, thus $\nabla w_{\lambda_0+\nu}^+$ on $ \Omega_{\lambda_0+\nu} $ for every $ 0 \le \nu < \bar\nu_0$. The latter proves that $u\leq u_{\lambda_0+\nu}$ in
$\Omega_{\lambda_0+\nu}\setminus R_{\lambda_0+\nu}(\Gamma)$ for every
$0<\nu<\bar\nu_0$. Such a contradiction shows that
\[
\lambda_0=0\,.
\]
}

\emph{Step 3: conclusion.} Since the moving plane procedure can be
performed in the same way but in the opposite direction, then this
proves the desired symmetry result.  The fact that the solution is
increasing in the $x_1$-direction in $\{x_1<0\}$ is implicit in the
moving plane procedure. Since $u$ has $C^1$ regularity, see Remark
\ref{standellest}, the fact that $u_{x_1}$ is positive for $x_1<0$
follows by the maximum principle, the H\"opf
lemma and the assumption $(h_f)$.

\end{proof}

\section{Proof of Theorem \ref{main2} and Corollary \ref{CorSim}}\label{mainproof}

\begin{proof}[Proof of Theorem \ref{main2}]

We first note that, thanks to a well-known result of Brezis and Kato
\cite{BK} and standard elliptic estimates (see also \cite{Stru}),
the solution $u$ is smooth in $\R^n \setminus \Gamma$.
Furthermore we observe that it is enough to
prove the theorem for the special case in which
{\textit {the origin does not belong to}} $\Gamma$. Indeed, if the
result is true in this special case, then we can apply it to the
function $u_z(x) := u(x+z),$ where $ z \in \{x_1=0\} \setminus
\Gamma \neq \emptyset$, which satisfies the equation
\eqref{problemubdd} with $\Gamma$ replaced by $-z + \Gamma$ (note
that $-z + \Gamma$ is a closed and proper subset of $\{x_1=0\}$ with
$\underset{\R^n}{\operatorname{Cap}_2}(-z +\Gamma)=0$ and such that
the origin does not belong to it).

\noindent Under this assumption, we consider the map $K: \R^n \setminus \{0\} \longrightarrow \R^n  \setminus \{0\} $ defined by $K = K(x):=\frac{x}{|x|^2}$. Given  $u$ solution to \eqref{problemubdd},
its Kelvin transform is given by
\begin{equation}\label{kelv} v(x):=\frac{1}{|x|^{n-2}} u\left(\frac{x}{|x|^2}\right), \quad
x \in \R^n \setminus \{\Gamma^* \cup \{0\}\},
\end{equation}
where $\Gamma^* = K( \Gamma)$. It follows that $v$ weakly
satisfies \eqref{problemubdd} in $\R^n
\setminus \{\Gamma^* \cup \{0\}\}$ and that $\Gamma^*\subset
\{x_1=0\}$ since, by assumption, $\Gamma\subset \{x_1=0\}$.
Furthermore, we also have that $\Gamma^*$ is bounded (not
necessarily closed) since we assumed that $0\notin\Gamma$.

The proceed further we need the following lemmata

\begin{lem}\label{diffeo} Let $F: \R^n \setminus \{0\} \longrightarrow \R^n \setminus \{0\}$
be a $C^1-$diffeomorphism and let $A$ be a bounded open set of $\R^n \setminus \{0\}$. If $C \subset A$ is a compact set such that
\begin{equation}\label{closedsetnullcap}
\underset{A}{\operatorname{Cap}_2}(C)=0,
\end{equation}
then
\begin{equation}\label{kelvclosedsetnullcap}
\underset{F(A)}{\operatorname{Cap}_2}(F(C))=0.
\end{equation}
\end{lem}

\begin{proof}
By hypothesis \eqref{closedsetnullcap} and by definition of
2-capacity, for every $\varepsilon
> 0$ let $\varphi_\varepsilon \in C^\infty_c(A)$ such that
\begin{itemize}
\item[a.] $\displaystyle \int_{A} | \nabla
\varphi_{\varepsilon}|^2 \,dx < \varepsilon$
\item[b.] $\displaystyle \varphi_{\varepsilon} \geq 1$ in a neighborhood $\mathcal{B}_{\varepsilon}$ of $C$.
\end{itemize}
Let $\psi_\varepsilon := \varphi_\varepsilon \circ G$, where
$G:=F^{-1}$. By definition of $\psi_\varepsilon$, we immediately
have that $\psi_\varepsilon \geq 1$ in a neighborhood
$\mathcal{B}^{'}_{\varepsilon}$ of the compact set $F(C)$. Moreover
\begin{equation}\nonumber
\begin{split}
\int_{F(A)} |\nabla \psi_\varepsilon(y)|^2 \, dy &= \int_{F(A)}
| JG (y_1,...,y_n) \cdot \nabla \varphi_\varepsilon (G_1(y_1),...,G_N(y_n))|^2 \, dy_1 \cdots dy_n \\
& \leq \int_{F(A)} \|JG \|_{\infty, {\overline{F(A)}}} |\nabla \varphi_\varepsilon (G_1(y_1),...,G_N(y_n))|^2\, dy_1 \cdots dy_n \\
& \leq C(F,A) \int_{F(A)} |\nabla \varphi_\varepsilon (G_1(y_1),...,G_N(y_n))|^2\, dy_1 \cdots dy_n \\
& =  C(F,A) \int_A |\nabla \varphi_\varepsilon (x_1,...,x_n)|^2 |\det(JF (x_1,...,x_n))| dx_1 \cdots dx_n\\
&\leq {\tilde C}(F,A) \int_A |\nabla \varphi_\varepsilon|^2 \, dx <{\tilde C}(F,A) \varepsilon.\\
\end{split}
\end{equation}
Since ${\tilde C}(F,A) $ is independent of $\varepsilon$, the
desired conclusion follows at once.

\end{proof}

\begin{lem}\label{geom1} Let $\Gamma $ be a closed subset of $\R^n$, with $ n \ge 3$. Also suppose that $0 \not\in \Gamma$ and
\begin{equation}\label{gamm}
\underset{\R^n}{\operatorname{Cap}_2}(\Gamma)=0.
\end{equation}
Then
\begin{equation}\label{gamstar} \underset{\R^n}{\operatorname{Cap}_2}(\Gamma^*)=0.
\end{equation}
\end{lem}

\begin{proof}
Since $0$ belongs to the open set $\R^n \setminus \Gamma$, there exists $r_0 \in (0,1) $ such that $ B_{r_0}(0) \cap \Gamma = \emptyset$. Therefore, $\displaystyle \Gamma= \bigcup_{m=1}^{+\infty} \left[\Gamma \cap
(\overline{B_m(0)} \setminus B_{r_0}(0))\right]$ and so
$$\underset{\R^n}{\operatorname{Cap}_2}\left[\Gamma \cap
(\overline{B_m(0)} \setminus B_{r_0}(0))\right]=0, \, \, \, \forall m \in \N,$$
since \eqref{gamm} is in force. The latter and $ n \ge 3$ imply that
$$\underset{A_m} {\operatorname{Cap}_2}\left[\Gamma \cap
(\overline{B_m(0)} \setminus B_{r_0}(0))\right]=0, \, \, \, \forall m \in \N,$$
where $ A_m := {B_{m+1}(0)} \setminus {\overline {B_{\frac{r_0}{2}}(0))}}$ is an open and bounded set for every $ m \ge 1$.
An application of lemma \ref{diffeo} with $F =K$, the inversion $x \to \frac{x}{\vert x \vert^2}$, $ A =A_m$ and $C = \Gamma \cap (\overline{B_m(0)} \setminus B_{r_0}(0))$ yields
$$\underset{K(A_m)}{\operatorname{Cap}_2} K \left(\Gamma \cap
(\overline{B_m(0)} \setminus B_{r_0}(0))\right)=0, \, \, \, \forall m \in \N$$
and so
$$\underset{\R^n}{\operatorname{Cap}_2} K \left(\Gamma \cap
(\overline{B_m(0)} \setminus B_{r_0}(0))\right)=0, \, \, \, \forall m \in \N.$$
But $\displaystyle \Gamma^*= K(\Gamma) = K(\bigcup_{m=1}^{+\infty} \left[\Gamma \cap
(\overline{B_m(0)} \setminus B_{r_0}(0))\right]) = \bigcup_{m=1}^{+\infty} K \left(\Gamma \cap
(\overline{B_m(0)} \setminus B_{r_0}(0))\right) $ and the 2-capacity is an exterior measure (see e.g. \cite{evans}), so
the desired conclusion \eqref{gamstar} follows.

\end{proof}

\noindent  Let us now fix some notations. We set
\begin{equation}\label{eq:sn2hgjhsdgf}
\Sigma_\lambda=\{x\in\mathbb{R}^n\,:\,x_1 <\lambda\}\,.
\end{equation}
As above $x_\lambda=(2\lambda-x_1,x_2,\ldots,x_n)$ is the reflection
of $x$ through the hyperplane $T_\lambda=\{x=(x_1,...,x_n)\in \R^n \
| \ x_1=\lambda\}$. Finally we consider the Kelvin transform $v$ of
$u$ defined in \eqref{kelv} and we set
\begin{equation}\label{mov}
w_{\lambda}(x)=v(x)-v_\lambda (x)= v(x) - v(x_\lambda).
\end{equation}
Note that $v$ weakly solves
\begin{equation}\label{debil122bissj}
\int_{\mathbb{R}^n} \nabla v\nabla
\varphi\,dx\,=\,\int_{\mathbb{R}^n} v^{2^*-1}
\varphi\,dx\qquad\forall \varphi\in C^{1}_c(\mathbb{R}^n\setminus
\Gamma^* \cup \{0\})\,.
\end{equation}
 and $v_\lambda$ weakly solves
\begin{equation}\label{debil122biss}
\int_{\mathbb{R}^n} \nabla v_\lambda\nabla
\varphi\,dx\,=\,\int_{\mathbb{R}^n} v_\lambda^{2^*-1}
\varphi\,dx\qquad\forall \varphi\in C^{1}_c(\mathbb{R}^n\setminus
R_\lambda(\Gamma^* \cup \{0\}))\,.
\end{equation}

The properties of the Kelvin transform, the fact that
$0\notin\Gamma$ and the regularity of $u$ imply that $\vert v(x)
\vert \le C \vert x \vert^{2-N} $ for every $x \in \R^n$ such that
$\vert x \vert \ge R$, where $C$ and $R$ are positive constants
(depending on $u$).  In particular, for every $ \lambda <0$, we have
\begin{equation}\label{kelvinp}
v\in L^{2^*}(\Sigma_\lambda)\cap L^{\infty}(\Sigma_\lambda) \cap C^0 (\overline{\Sigma_\lambda}) \,.
\end{equation}

\begin{lem}\label{stimgrad} Under the assumption of Theorem
\ref{main2}, for every $\lambda < 0$, we have that $ w_\lambda^+ \in
L^{2^*}(\Sigma_\lambda), \nabla w_\lambda^+ \in L^2(\Sigma_\lambda)
$ and
\begin{equation}\label{buonastima}
\Vert w_\lambda^+ \Vert^2_{L^{2^*}(\Sigma_\lambda)}
\le C_S^2 \int_{\Sigma_\lambda}|\nabla w_\lambda^+|^2\,dx\leq
2 C_S^2 \frac{n+2}{n-2} \|v\|^{2^*}_{L^{2^*}(\Sigma_\lambda)},
\end{equation}
where $C_S$ denotes de best constant in Sobolev embedding.
\end{lem}

\begin{proof}
We immediately see that $w_\lambda^+ \in L^{2^*}(\Sigma_\lambda),$
since $0\leq w_\lambda^+\leq v \in L^{2^*}(\Sigma_\lambda)$ . The
rest of the proof follows the lines of the one of lemma
\ref{leaiuto2}. Arguing as in section 2, for every $ \varepsilon>0$,
we can find a function $\psi_\varepsilon \in C^{0,1}(\R^N, [0,1])$
such that
$$\int_{\Sigma_\lambda} |\nabla \psi_\varepsilon|^2 < 4 \varepsilon$$
and $\psi_\varepsilon = 0$ in an open neighborhood
$\mathcal{B_{\varepsilon}}$ of $R_{\lambda}(\{\Gamma^* \cup
\{0\}\})$, with $\mathcal{B_{\varepsilon}} \subset
\Sigma_{\lambda}$.

Fix $ R_0>0$ such that $R_{\lambda}(\{\Gamma^* \cup \{0\} )\subset
B_{R_0} $ and, for every $ R > R_0$, let $\varphi_R$ be a standard
cut off function such that $ 0 \le \varphi_R \le 1 $ on $ \R^n$,
$\varphi_R=1$ in $B_R$, $\varphi_R=0$ outside $B_{2R}$ with
$|\nabla\varphi_R|\leq 2/R,$ and consider
$$\varphi\,:= \begin{cases}
\, w_\lambda^+\psi_\varepsilon^2\varphi_R^2 \, & \text{in}\quad\Sigma_\lambda, \\
0 &  \text{in}\quad \R^n \setminus  \Sigma_\lambda.
\end{cases}$$
Now, as in Lemma \ref{leaiuto} we see that $ \varphi \in C^{0,1}_c(\R^n)$ with $supp(\varphi)$ contained in $\overline {\Sigma_{\lambda} \cap B_{2R}} \setminus R_{\lambda}(\{\Gamma^* \cup \{0\}\})$ and
\begin{equation}\label{gradvarphi-intero}
\nabla \varphi = \psi_\varepsilon^2 \varphi_R^2 (\nabla w_\lambda \chi_{supp(w_\lambda^+) \cap supp(\varphi)}) +
2 (w_\lambda^+ \chi_{supp(\varphi)}) (\psi_\varepsilon^2 \varphi_R \nabla \varphi_R + \psi_\varepsilon \varphi_R^2  \nabla  \psi_\varepsilon).
\end{equation}
Therefore, by a standard density argument, we can use $\varphi$ as test function in \eqref{debil122bissj} and in \eqref{debil122biss} so that, subtracting we get
\begin{equation}\label{djfsdjfbskfhasklfh}
\begin{split}
\int_{\Sigma_\lambda}|\nabla w_\lambda \chi_{supp(w_\lambda^+)}|^2 \psi_\varepsilon^2\varphi_R^2\,dx&=
-2\int_{\Sigma_\lambda}\nabla w_\lambda \nabla \psi_\varepsilon
w_\lambda^+ \psi_\varepsilon\varphi_R^2\,dx
-2\int_{\Sigma_\lambda} \nabla w_\lambda \nabla \varphi_R w_\lambda^+ \varphi_R \psi_\varepsilon^2\,dx\\
&+\int_{\Sigma_\lambda} (v^{2^*-1}-v_\lambda^{2^*-1}) w_\lambda^+ \psi_\varepsilon^2\varphi_R^2\,dx\,\,\\
&=:\,I_1+I_2+I_3\,.
\end{split}
\end{equation}
Exploiting also Young's inequality and recalling that
$0\leq w_\lambda^+\leq v$, we get that
\begin{equation}\label{I1}
\begin{split}
|I_1|&\leq \frac{1}{4} \int_{\Sigma_\lambda}|\nabla w_\lambda
\chi_{supp(w_\lambda^+)}|^2\psi_\varepsilon^2\varphi_R^2\,dx
+4\int_{\Sigma_\lambda}| \nabla \psi_\varepsilon|^2(w_\lambda^+)^2\varphi_R^2\,dx\\
&\leq  \frac{1}{4} \int_{\Sigma_\lambda}|\nabla w_\lambda
\chi_{supp(w_\lambda^+)}|^2\psi_\varepsilon^2\varphi_R^2\,dx
+ 16 \varepsilon \|v\|^2_{L^\infty(\Sigma_\lambda)} .\\
\end{split}
\end{equation}

Furthermore we have that
\begin{equation}\label{I2}
\begin{split}
|I_2|&\leq \frac{1}{4} \int_{\Sigma_\lambda}|\nabla w_\lambda
\chi_{supp(w_\lambda^+)}|^2\psi_\varepsilon^2\varphi_R^2\,dx
+4\int_{\Sigma_\lambda\cap(B_{2R}\setminus B_{R})}|\nabla \varphi_R|^2(w_\lambda^+)^2 \psi_\varepsilon^2\,dx\\
&\leq  \frac{1}{4} \int_{\Sigma_\lambda}|\nabla w_\lambda
\chi_{supp(w_\lambda^+)}|^2\psi_\varepsilon^2\varphi_R^2\,dx\\
&+ 4\left(\int_{\Sigma_\lambda\cap(B_{2R}\setminus B_{R})}|\nabla
\varphi_R|^n\,dx\right)^{\frac{2}{n}}
\left(\int_{\Sigma_\lambda\cap(B_{2R}\setminus B_{R})}v^{2^*}\,dx\right)^{\frac{n-2}{n}}\\
&\leq  \frac{1}{4} \int_{\Sigma_\lambda}|\nabla w_\lambda
\chi_{supp(w_\lambda^+)}|^2\psi_\varepsilon^2\varphi_R^2\,dx\,+\,
c(n) \left(\int_{\Sigma_\lambda\cap(B_{2R}\setminus
B_{R})}v^{2^*}\,dx\right)^{\frac{n-2}{n}}
\end{split}
\end{equation}
where $c(n)$ is a positive constant depending only on the dimension $n$.

Let us now estimate $I_3$. Since $v(x), v_\lambda(x)>0$, by the
convexity of $ t \to t^{2^*-1},$ for $ t >0$, we obtain
$v^{2^*-1}(x)-v_\lambda^{2^*-1}(x) \le  \frac{n+2}{n-2}
v_\lambda^{2^*-2}(x) (v(x) - v_\lambda(x))$, for every $x \in
\Sigma_{\lambda}$. Thus, by making use of the monotonicity of $ t
\to t^{2^*-2}$, for $ t >0$ and the definition of $w_\lambda^+$ we
get  $(v^{2^*-1}-v_\lambda^{2^*-1})w_\lambda^+ \le  \frac{n+2}{n-2}
v_\lambda^{2^*-2}(v-v_\lambda) w_\lambda^+ \le \frac{n+2}{n-2}
v^{2^*-2}(w_\lambda^+)^2.$ Therefore
\begin{equation}\label{I3}
\begin{split}
I_3& \leq
\frac{n+2}{n-2} \int_{\Sigma_\lambda}
v^{2^*-2}(w_\lambda^+)^2 \psi_\varepsilon^2\varphi_R^2\,dx\,\\
&\leq \frac{n+2}{n-2} \int_{\Sigma_\lambda} v^{2^*-2}v^2 dx\,=
\frac{n+2}{n-2} \int_{\Sigma_\lambda} v^{2^*}\,dx = \frac{n+2}{n-2} \|v\|^{2^*}_{L^{2^*}(\Sigma_\lambda)}\\
\end{split}
\end{equation}where we also used that $0\leq w_\lambda^+\leq v$.
Taking into account the estimates on $I_1$, $I_2$ and $I_3$, by \eqref{djfsdjfbskfhasklfh} we deduce that

\begin{equation}\label{primastima}
\begin{split}
\int_{\Sigma_\lambda}|\nabla w_\lambda
\chi_{supp(w_\lambda^+)}|^2\varphi_\varepsilon^2\varphi_R^2\,dx
&\leq 32 \varepsilon \|v\|^2_{L^\infty(\Sigma_\lambda)}\\
&+ 2c(n) \left(\int_{\Sigma_\lambda\cap(B_{2R}\setminus
B_{R})}v^{2^*}\,dx\right)^{\frac{n-2}{n}}\\
& + 2\frac{n+2}{n-2} \|v\|^{2^*}_{L^{2^*}(\Sigma_\lambda)}.
\end{split}
\end{equation}
By Fatou Lemma, as $\varepsilon$ tends to zero and $R$ tends to
infinity,  we deduce that $\nabla w_\lambda \chi_{supp(w_\lambda^+)}
\in L^2(\Sigma_{\lambda})$.  We also note that $ \varphi \to
w_\lambda^+$ in $L^{2^*}(\Sigma_\lambda)$, by definition of
$\varphi$, and that $\nabla \varphi \to \nabla w_\lambda
\chi_{supp(w_\lambda^+)}$ in $L^2(\Sigma_{\lambda})$, by
\eqref{gradvarphi-intero} and the fact that $w_\lambda^+\in
L^{2^*}(\Sigma_\lambda)$.  Therefore, $\nabla w_\lambda
\chi_{supp(w_\lambda^+)}$ is the distributional gradient of $\nabla
w_\lambda^+$ and so $ \nabla w_\lambda^+$ in $L^2(\Sigma_\lambda)$
with (taking limit in \eqref{primastima})
\begin{equation}\label{primastima2}
\begin{split}
\int_{\Sigma_\lambda}|\nabla w_\lambda|^2 \,dx &\leq 2
\frac{n+2}{n-2} \|v\|^{2^*}_{L^{2^*}(\Sigma_\lambda)}.
\end{split}
\end{equation}
Since $ \varphi \in C^{0,1}_c(\R^n)$ we also have
\begin{equation}\label{secondastima}
\begin{split}
\Big( \int_{\Sigma_\lambda} \varphi^{2^*} \Big)^{\frac{2}{2^*}} \le C_S^2 \int_{\Sigma_\lambda} \vert \nabla \varphi \vert^2
\end{split}
\end{equation}
where $ C_S$ denotes de best constant in Sobolev embedding.  Thus,
passing to the limit in \eqref{secondastima} and using the above
convergence results, we get the desired conclusion
\eqref{buonastima}.

\end{proof}

We can now complete the proof of Theorem \ref{main2}. As for the proof of Theorem \ref{main},  we split the proof into three steps and we start with

\noindent \emph{Step 1: there exists $M>1$ such that $v \leq v_\lambda$ in $\Sigma_\lambda\setminus
R_\lambda(\Gamma^* \cup \{0\})$, for all $\lambda< -M$.}

Arguing as in the proof of Lemma \ref{stimgrad} and using the same notations and the same construction
for $\psi_\varepsilon$, $\varphi_R$ and $\varphi$, we get

\begin{equation}\label{kdfbdjklfbfkfakjl}
\begin{split}
\int_{\Sigma_\lambda}|\nabla w_\lambda^+|^2\varphi_\varepsilon^2\varphi_R^2\,dx&=
-2\int_{\Sigma_\lambda}\nabla w_\lambda^+\nabla \varphi_\varepsilon
w_\lambda^+ \varphi_\varepsilon\varphi_R^2\,dx
-2\int_{\Sigma_\lambda} \nabla w_\lambda^+\nabla \varphi_R w_\lambda^+ \varphi_R\varphi_\varepsilon^2\,dx\\
&+\int_{\Sigma_\lambda} (v^{2^*-1}-v_\lambda^{2^*-1}) w_\lambda^+ \varphi_\varepsilon^2\varphi_R^2\,dx\,\,\\
&=:\,I_1+I_2+I_3\,,
\end{split}
\end{equation}
where $I_1,I_2$ and $I_3$ can be estimated exactly as in \eqref{I1},
\eqref{I2} and \eqref{I3}. The latter yield
\begin{equation}
\begin{split}
\int_{\Sigma_\lambda}|\nabla
w_\lambda^+|^2\varphi_\varepsilon^2\varphi_R^2\,dx &\leq 32
\varepsilon \|v\|_{L^\infty(\Sigma_\lambda)}^2\\
&+ 2c(n) \left(\int_{\Sigma_\lambda\cap(B_{2R}\setminus
B_{R})}v^{2^*}\,dx\right)^{\frac{n-2}{n}}\\
&+ 2 \frac{n+2}{n-2} \int_{\Sigma_\lambda} v^{2^*-2}(w_\lambda^+)^2
\psi_\varepsilon^2\varphi_R^2.
\end{split}
\end{equation}
Taking the limit in the latter, as $\varepsilon$ tends to zero and $R$ tends to infinity, leads to
\begin{equation}
\begin{split}
\int_{\Sigma_\lambda}|\nabla w_\lambda^+|^2 \,dx &\leq 2 \frac{n+2}{n-2} \int_{\Sigma_\lambda} v^{2^*-2}(w_\lambda^+)^2  < +\infty
\end{split}
\end{equation}
which combined with Lemma \ref{stimgrad} gives
\begin{equation}\label{kushfkusfksf}
\begin{split}
\int_{\Sigma_\lambda}|\nabla w_\lambda^+|^2\,dx&\leq 2 \frac{n+2}{n-2}
\int_{\Sigma_\lambda} v^{2^* -2}(w_\lambda^+)^2 dx\,\,\\
&\leq 2 \frac{n+2}{n-2} \left (\int_{\Sigma_\lambda} v^{2^*}
\,dx\right)^{\frac{2}{n}} \left(\int_{\Sigma_\lambda}
(w_\lambda^+)^{2^*}\,dx\right)^{\frac{2}{2^*}}\\
&\leq 2 \frac{n+2}{n-2}C_S^2 \left (\int_{\Sigma_\lambda} v^{2^*}
\,dx\right)^{\frac{2}{n}} \left(\int_{\Sigma_\lambda} |\nabla
w_\lambda^+|^2\,dx\right).
\end{split}
\end{equation}

Recalling that $v\in L^{2^*}(\Sigma_\lambda)$, we deduce the existence of $ M>1$ such that
\[
2 \frac{n+2}{n-2}C_S^2\left (\int_{\Sigma_\lambda} v^{2^*}
\,dx\right)^{\frac{2}{n}}<1\,
\]
for every $\lambda < -M$. The latter and \eqref{kushfkusfksf} lead
to
\[
\int_{\Sigma_\lambda}|\nabla w_\lambda^+|^2\,dx=0\,.
\]
This implies that $ w_\lambda^+=0$ by Lemma \ref{stimgrad} and the
claim is proved.\\

\noindent To proceed further we
 define
\begin{equation}\nonumber
\Lambda_0=\{\lambda<0 : v\leq
v_{t}\,\,\,\text{in}\,\,\,\Sigma_t\setminus R_t(\Gamma^* \cup
\{0\})\,\,\,\text{for all $t\in(-\infty,\lambda]$}\}
\end{equation}
and
\begin{equation}\nonumber
\lambda_0=\sup\,\Lambda_0.
\end{equation}

\noindent \emph{Step 2: we have that $\lambda_0=0$.} We argue by
contradiction and suppose that $\lambda_0<0$. By continuity we know
that $v\leq v_{\lambda_0}$ in $\Sigma_{\lambda_0}\setminus
R_{\lambda_0}(\Gamma^* \cup \{0\})$. By the strong maximum principle
we deduce that $v< v_{\lambda_0}$ in $\Sigma_{\lambda_0}\setminus
R_{\lambda_0}(\Gamma^* \cup \{0\})$. Indeed, $v= v_{\lambda_0}$ in
$\Sigma_{\lambda_0}\setminus R_{\lambda_0}(\Gamma^* \cup \{0\})$ )
is not possible if $\lambda_0<0$, since in this case $v$ would be
singular somewhere on $R_{\lambda_0}(\Gamma^* \cup \{0\})$. Now, for
some $\bar\tau>0$, that will be fixed later on, and for any
$0<\tau<\bar\tau$ we show that $v\leq v_{\lambda_0+\tau}$ in
$\Sigma_{\lambda_0+\tau}\setminus R_{\lambda_0+\tau}(\Gamma^* \cup
\{0\})$ obtaining a contradiction with the definition of $\lambda_0$
and proving thus the claim. To this end we are going to show that,
for every $ \delta>0$ there are $ \bar{\tau}(\delta, \lambda_0)>0 $
and a compact set $K$ (depending on $\delta$ and $\lambda_0$) such
that
\[
K \subset \Sigma_{\lambda}\setminus R_{\lambda}(\Gamma^* \cup \{0\}), \qquad
\int_{\Sigma_{\lambda}\setminus K}\,v^{2^*} < \delta, \qquad \forall \,  \lambda \in [\lambda_0, \lambda_0 + \bar{\tau}].
\]
To see this, we note that for every  every $ \delta>0$ there are $
\tau_1(\delta, \lambda_0)>0 $ and a compact set $K$ (depending on
$\delta $ and $ \lambda_0$) such that $\displaystyle
\int_{\Sigma_{\lambda_0}\setminus K}\,v^{2^*} < \frac{\delta}{2}$
and $K \subset \Sigma_{\lambda}\setminus R_{\lambda}(\Gamma^* \cup
\{0\})$ for every $ \lambda \in [\lambda_0, \lambda_0 + \tau_1].$
Consequently $u$ and $u_{\lambda}$ are well defined on $K$ for every
$ \lambda \in [\lambda_0, \lambda_0 + \tau_1].$ Hence, by the
uniform continuity of the function $ g(x,\lambda) := u(x) -
u(2\lambda-x_1,x^{'}) $ on the compact set $K \times [\lambda_0,
\lambda_0 + \tau_1]$ we can ensure that  $K \subset
\Sigma_{{\lambda_0+\tau}} \setminus R_{{\lambda_0+\tau}}(\Gamma^*
\cup \{0\})$ and $u< u_{\lambda_0+\tau}$ in $K$ for any $0 \le \tau<
\tau_2$, for some $\tau_2= \tau(\delta,\lambda_0) \in (0, \tau_1)$.
Clearly we can also assume that $\tau_2< \frac{\vert \lambda_0
\vert}{4}.$ Finally, since $v^{2^*} \in L^1(\Sigma_{\lambda_0 +
\frac{\vert \lambda_0 \vert}{4}})$ and $\displaystyle
\int_{\Sigma_{\lambda_0}\setminus K}\,v^{2^*} < \frac{\delta}{2}$,
we obtain the existence of $\bar{\tau} \in (0, \tau_2)$ such that
$\displaystyle \int_{\Sigma_{\lambda}\setminus K} \, v^{2^*} <
\delta$ for all $ \lambda \in [\lambda_0, \lambda_0 + \bar{\tau}]$.

Now we repeat verbatim the arguments used in the proof of Lemma \ref{stimgrad} but using the test function
$$\varphi \,:= \begin{cases}
\, w_{\lambda_0+\tau}^+\psi_\varepsilon^2\varphi_R^2 \, & \text{in}\quad\Sigma_{\lambda_0+\tau}  \\
 0 &  \text{in}\quad \R^n \setminus \Sigma_{\lambda_0+\tau}.
\end{cases}$$

Thus we recover the first inequality in \eqref{kushfkusfksf}, which
immediately gives, for any $0 \le \tau<\bar\tau$
\begin{equation}\label{asdgcahijgljhp}
\begin{split}
\int_{\Sigma_{\lambda_0+\tau}\setminus K}|\nabla w_{\lambda_0+\tau}^+|^2\,dx
&\leq 2 \frac{n+2}{n-2} \int_{\Sigma_{\lambda_0+\tau}\setminus K} v^{2^* -2} (w_{\lambda_0+\tau}^+)^2\,dx\\
&\leq 2 \frac{n+2}{n-2}\left (\int_{\Sigma_{\lambda_0+\tau}\setminus K}
v^{2^*} \,dx\right)^{\frac{2}{n}}
\left(\int_{\Sigma_{\lambda_0+\tau}\setminus K} (w_{\lambda_0+\tau}^+)^{2^*}\,dx\right)^{\frac{2}{2^*}}\\
&\leq 2 \frac{n+2}{n-2} C_S^2 \left (\int_{\Sigma_{\lambda_0+\tau}\setminus
K} v^{2^*} \,dx\right)^{\frac{2}{n}}
\left(\int_{\Sigma_{\lambda_0+\tau}\setminus K} |\nabla
w_{\lambda_0+\tau}^+|^2\,dx\right)
\end{split}
\end{equation}
since $ w_{\lambda_0+\tau}^+$ and $\nabla w_{\lambda_0+\tau}^+$ are zero in a neighbourhood of $K$, by the above construction.
Now we fix $\delta < \frac{1}{2} [2 \frac{n+2}{n-2} C_S^2]^{-\frac{n}{2}}$ and we observe that with this choice we have
\[
2 \frac{n+2}{n-2} C_S^2 \left (\int_{\Sigma_{\lambda_0+\tau}\setminus K}
v^{2^*} \,dx\right)^{\frac{2}{n}}< \frac{1}{2}, \qquad \forall \,\, 0 \le \tau<\bar\tau
\]
which plugged into \eqref{asdgcahijgljhp} implies that
$\displaystyle \int_{\Sigma_{\lambda_0+\tau}\setminus K}|\nabla
w_{\lambda_0+\tau}^+|^2\,dx = 0$ for every $0 \le \tau<\bar\tau$.
Hence $\displaystyle \int_{\Sigma_{\lambda_0+\tau}}|\nabla
w_{\lambda_0+\tau}^+|^2\,dx = 0$ for every $0 \le \tau<\bar\tau$,
since $\nabla w_{\lambda_0+\tau}^+$ is zero in a neighborhood of
$K$. The latter and Lemma \ref{stimgrad} imply that $
w_{\lambda_0+\tau}^+ =0 $ on $ \Sigma_{\lambda_0 +\tau}$ for every
$0 \le \tau<\bar\tau$ and thus $v\leq v_{\lambda_0+\tau}$ in
$\Sigma_{\lambda_0+\tau}\setminus R_{\lambda_0+\tau}(\Gamma^* \cup
\{0\})$ for every $0 \le \tau<\bar\tau$ . Which proves the claim of
Step 2.

\noindent \emph{Step 3: conclusion.} The symmetry of the Kelvin
transform $v$ follows now performing the moving plane method in the
opposite direction. The fact that that $v$ is symmetric w.r.t. the
hyperplane $\{ x_1=0 \}$ implies the symmetry of the solution $u$
w.r.t. the hyperplane $\{ x_1=0 \}$. The last claim then follows by
the invariance of the considered problem with respect to isometries
(translations and rotations).

\end{proof}

\begin{proof}[Proof of Corollary \ref{CorSim}]

The function $v(x)=u(x+x_0)$ satisfies the assumptions of Theorem
\ref{main2} with $\Gamma = \{0\}$. An application of  Theorem
\ref{main2} yields that $v$ is symmetric with respect to every
hyperplane through the origin and so the original solution $u$ must
be radially symmetric with respect to $x_0$.  This proves item (i).
Since item (ii) is a special case of item (iii) with $k=1$, we need
only to prove item (iii). To this end we observe that, up to an
isometry, we can suppose that the affine
$k-$dimensional subspace is $\{ x_{k+1} =...= x_n = 0 \}$.
Therefore, we can apply Theorem \ref{main2} to get that $u$ is
symmetric with respect to each hyperplane of $\R^n$ containing $\{
x_{k+1} =...= x_n = 0 \}$; i.e., $u$ is invariant with respect to
every rotation of $ \R^n$ which leaves invariant the set $\{ x_{k+1}
=...= x_n = 0 \}$. Note that we can apply Theorem \ref{main} since
any affine $k-$dimensional subspace of $\R^n$, with $ 1 \le k \le
n-2$, has zero 2-capacity in $\R^n$ (and so
$\underset{\R^n}{\operatorname{Cap}_2}(\Gamma)=0$).

\end{proof}

\bigskip


\begin{thebibliography}{99}
\bibitem{A} {\sc A.D.~Alexandrov}:
\newblock A characteristic property of the spheres.
\newblock { \em Ann. Mat. Pura Appl.} 58, 1962, pp.  303 -- 354.


\bibitem{BN} {\sc H.~Berestycki, L.~Nirenberg},
\newblock On the method of moving planes and the sliding method.
\newblock {\em Bulletin Soc. Brasil. de Mat Nova Ser}, 22(1), 1991, pp. 1--37.


\bibitem{boccardo}{\sc L.~Boccardo,  L.~Orsina}.
\newblock Semilinear elliptic equations with singular nonlinearities.
\newblock  {\em Calc. Var. Partial Differential Equations}, 37(3-4), 2010, pp. 363--380.

\bibitem{BK} {\sc H. Brezis, T.Kato}:
\newblock Remarks on the Schrodinger operator with singular complex potentials.
\newblock {\em J. Math. Pures Appl.} 58 (1979) pp. 137-151


\bibitem{CLN2} {\sc L. Caffarelli, Y.Y. Li, L.~Nirenberg},
\newblock  Some remarks on singular solutions of nonlinear elliptic equations. II: Symmetry and monotonicity via moving planes.
\newblock \emph{Advances in geometric analysis}, 97–105, Adv. Lect. Math. (ALM), 21, Int. Press, Somerville, MA, 2012.

\bibitem{CanDeg}{\sc A.~Canino, M.~Degiovanni},
\newblock A variational approach to a class of singular semilinear elliptic equations.
\newblock  {\em J. Convex Anal.}, 11(1), 2004, pp. 147--162.



\bibitem{gras1}{\sc A.~Canino,  M. Grandinetti, B. Sciunzi}.
\newblock Symmetry of solutions of some semilinear elliptic equations with singular nonlinearities.
\newblock {\em J. Differential Equations},  255(12), 2013, pp. 4437--4447.

\bibitem{luigi}{\sc A.~Canino,  L. Montoro, B. Sciunzi}.
\newblock The moving plane method for singular semilinear elliptic problems.
\newblock {\em Nonlinear Analysis TMA},  preprint.



\bibitem{ccmcan} {\sc A. Canino, B. Sciunzi},
\newblock A uniqueness result for some singular semilinear elliptic equations
\newblock  {\em Comm. Contemporary Math.},  18(6),  2016, 1550084, 9 pp..




\bibitem{nodt} {\sc A. Canino,  B. Sciunzi, A. Trombetta}.
\newblock Existence and uniqueness for p-{L}aplace equations involving singular nonlinearities.
\newblock  {\em NoDEA Nonlinear Differential Equations Appl.}, 23,  2016, pp. 23:8.



\bibitem{crandall}{\sc M.G. Crandall, P.H. Rabinowitz, L. Tartar}.
\newblock On a Dirichlet problem with a singular nonlinearity.
\newblock  {\em Comm. P.D.E.}, 2, 1977, pp. 193--222.

\bibitem{evans}{\sc L.C. Evans, R.F. Gariepy}.
\newblock Measure theory and fine properties of fun tions. Studies in Advanced Mathematics. CRC Press,
Bo a Raton, FL, 1992.



\bibitem{GNN}{\sc B.~Gidas, W.~M.~Ni, L.~Nirenberg}.
\newblock Symmetry and related properties via the maximum principle.
\newblock  {\em Comm. Math. Phys.}, 68, 1979, pp. 209--243.


\bibitem{GT} {\sc D.~Gilbarg, N.~S.~Trudinger}.
\newblock Elliptic partial differential equations of second order. {\em Reprint of the 1998 Edition},
Springer.

\bibitem{lazer}{\sc A.C. Lazer, P.J. McKenna}.
\newblock On a singular nonlinear elliptic boundary-value problem.
\newblock  {\em Proc. AMS}, 111, 1991, pp. 721--730.


\bibitem{mps}{\sc L. Montoro, F. Punzo, B. Sciunzi}.
\newblock Qualitative properties of singular solutions to nonlocal problems.
\newblock {\em Preprint.}

\bibitem{oliva}{\sc F. Oliva, F. Petitta}.
\newblock On singular elliptic equations with measure sources.
\newblock {\em ESAIM Control Optim. Calc. Var.}, 22(1), 2016, pp. 289--308.



\bibitem{Dino} {\sc B. Sciunzi}.
\newblock On the moving Plane Method for singular solutions to semilinear elliptic equations.
\newblock {\em J. Math. Pures Appl.} (9), 108 (2017), no. 1, 111--123.


\bibitem{serrin} {\sc J. Serrin}.
\newblock A symmetry problem in potential theory.
\newblock {\em Arch. Rational Mech. Anal.}, 43, 1971, pp. 304--318.


\bibitem{Stru}{\sc M. Struwe}.
\newblock Variational methods. Applications to nonlinear partial differential equations and Hamiltonian systems. Fourth edition, 34. Springer-Verlag, Berlin, 2008.




\bibitem{stuart}{\sc C.A. Stuart}.
\newblock Existence and approximation of solutions of nonlinear elliptic equations.
\newblock  {\em Math. Z.}, 147, 1976, pp. 53--63.


\bibitem{T} {\sc S. Terracini}.
\newblock On positive entire solutions to a class of equations with a singular coefficient and critical exponent.
\newblock {\em Adv. Differential Equations}, 1, 1996, pp. 241--264.


\end{thebibliography}
\end{document}